%% file: root.tex
\documentclass[letterpaper, 10 pt, conference]{ieeeconf}  % Comment this line out if you need a4paper
\IEEEoverridecommandlockouts                              % This command is only needed if 
\usepackage{graphics} % for pdf, bitmapped graphics files
\usepackage{epsfig} % for postscript graphics files
\usepackage{todonotes}

\usepackage{amsmath,amssymb,amsfonts,amsthm}
\usepackage{xcolor}
\usepackage{cite}
\usepackage{mathtools}
\mathtoolsset{showonlyrefs}
\usepackage{comment}
\usepackage{array}
\usepackage{booktabs}
\usepackage{caption}
\usepackage{subcaption}
\usepackage[font=footnotesize,labelfont=bf]{caption}
\usepackage[hidelinks]{hyperref}
\usepackage{subcaption}
\usepackage{xurl} % nicer URL line breaks
\newtheorem{theorem}{Theorem}

\newtheorem{lemma}{Lemma}
\theoremstyle{definition}
\newtheorem{definition}{Definition}

\newtheorem{example}{Example}

\newtheorem{remark}{Remark}

\newtheorem{assumption}{Assumption}

\newtheorem{problem}{Problem}
\newtheorem{algorithm}{Algorithm}

\input{tyler_macros}
\renewcommand{\R}{\mathbb{R}}
\renewcommand{\FF}{\mathcal{F}}
\newcommand{\XX}{\mathcal{X}}
\DeclareMathOperator{\image}{im}
\renewcommand{\leq}{\leqslant}
\renewcommand{\geq}{\geqslant}

\definecolor{gatororange}{HTML}{FA4616}
\definecolor{buzzgold}{HTML}{EAAA00}
\definecolor{techgold}{HTML}{B3A369}
\definecolor{navyblue}{HTML}{003057}
\definecolor{techgreen}{HTML}{22EE66}

\newif\ifshowcomments
\showcommentstrue% Set to true to show comments by default
\title{\LARGE \bf
Asynchronous Nonlinear Sheaf Diffusion for Multi-Agent Coordination
}
\author{Yichen Zhao$^{\ast}$, Tyler Hanks$^{\ast}$, Hans Riess$^{\ast}$, Samuel Cohen, Matthew Hale, James Fairbanks % <-this % stops a space
\thanks{$^\ast$These authors contributed equally.}% <-this % stops a space
\thanks{Tyler Hanks and Samuel Cohen are with the Department of Computer and Information Science and Engineering and James Fairbanks is with the Department of Mechanical and Aerospace Engineering, University of Florida, emails:
        {\tt\small \{t.hanks,samuel.cohen, fairbanksj\}@ufl.edu}.
        Yichen Zhao, Hans Riess, and Matthew Hale are with the Department of Electrical and Computer Engineering, Georgia Institute of Technology, emails: 
        {\tt\small \{yzhao654,riess, mhale30\}@gatech.edu}}%. 
\thanks{Riess was supported by DARPA under grant HR00112530235. Hanks was supported by the National Science Foundation Graduate Research Fellowship Program under Grant No. DGE-1842473. Any opinions, findings, and conclusions or recommendations expressed in this material are those of the author(s) and do not necessarily reflect the views of the NSF. Hanks, Cohen, Hale, and Fairbanks were supported by DARPA grant HR00112220038. Hanks and Fairbanks were also partially supported by ONR grant N00014-23-1-2339.
Hale and Zhao were also supported by AFOSR under grants FA9550-19-1-0169 and FA9550-23-1-0120.}}%.

\begin{document}

\maketitle
\thispagestyle{empty}
\pagestyle{empty}

%%%%%%%%%%%%%%%%%%%%%%%%%%%%%%%%%%%%%%%%%%%%%%%%%%%%%%%%%%%%%%%%%%%%%%%%%%%%%%%%
\begin{abstract}
    Cellular sheaves and sheaf Laplacians provide a far-reaching generalization of graphs and graph Laplacians, resulting in a wide array of applications ranging from machine learning to multi-agent control. In the context of multi-agent systems, so called coordination sheaves provide a unifying formalism that models heterogeneous agents and coordination goals over undirected communication topologies, and applying sheaf diffusion drives agents to achieve their coordination goals. Existing literature on sheaf diffusion assumes that agents can communicate and compute updates synchronously, which is an unrealistic assumption in many scenarios where communication delays or heterogeneous agents with different compute capabilities cause disagreement among agents. To address these challenges, we introduce asynchronous nonlinear sheaf diffusion. Specifically, we show that under mild assumptions on the coordination sheaf and bounded delays in communication and computation, nonlinear sheaf diffusion converges to a minimizer of the Dirichlet energy of the coordination sheaf at a linear rate proportional to the delay bound. We further show that this linear convergence is attained from arbitrary initial conditions and the analysis depends on the spectrum of the sheaf Laplacian in a manner that generalizes the standard graph Laplacian case. We provide several numerical simulations to validate our theoretical results.
\end{abstract}

%%%%%%%%%%%%%%%%%%%%%%%%%%%%%%%%%%%%%%%%%%%%%%%%%%%%%%%%%%%%%%%%%%%%%%%%%%%%%%%%
\section{Introduction}~\label{sec:intro}
Multi-agent coordination is a fundamental problem in control and optimization~\cite{olfati-saber_consensus_2007}. A key difficulty arises from the heterogeneity of systems, where agents may have different capabilities, state spaces, or assigned roles~\cite{bao_recent_2022}. Further challenges arise in operational environments where communication is limited, such as wireless communication among low-flying UAVs~\cite{floreano_science_2015} (unmanned aerial vehicles) or communication among UUVs~\cite{akyildiz_underwater_2005} (unmanned undersea vehicles). 
In many of these environments, agents operate asynchronously, and they may both (a) generate information at different times and rates,
say, due to different clock frequencies~\cite{simeone_distributed_2008}, and (b) communicate information at different times and rates, 
due to environmental factors causing unexpected latency and throughput limitations~\cite{zeng_wireless_2016}.

There exists a wide range of networked multi-agent system models and applications, and this breadth has led to a large body of literature in which similar problems may require entirely new analyses to account for their particular features. To this end, the recently introduced unifying formalism of \emph{coordination sheaves}~\cite{hanks2025distributed} subsumes many existing graph-based coordination problems~\cite{mesbahi2010graph}. Roughly speaking, a coordination sheaf is a flexible data structure built upon a communication graph that assigns different spaces to both agents (nodes) and their interactions (edges). 
The spaces attached to nodes model the states or decision variables of the agents. 
The spaces attached to edges model measurements or communications between pairs of agents. 
This structure allows agents to have different goals and even for a single agent to pursue different goals with different neighbors capturing general forms of agreement and disagreement. 

Existing algorithms for coordination sheaves rely on synchronous computation and communication to produce new iterates. 
This degree of synchronization is not always possible in practical applications as agents
may compute new iterates and communicate them asynchronously leading to delays in when information is generated and shared
among the agents~\cite{bertsekas1989parallel}.  Such delays cause agents to use outdated and inconsistent values of other agents' iterates in their computations. To address these challenges in the context of asynchronous multi-agent coordination, our primary contribution in this paper is to bring the coordination sheaves framework into the asynchronous setting.

Specifically, we develop an asynchronous multi-agent coordination algorithm for finding the minima of a Dirichlet energy function whose gradient is a nonlinear sheaf Laplacian~\cite{hansen2021opinion}, a generalization of the graph Laplacian to the sheaf-theoretic setting. We show that, under mild conditions on agents' coordination sheaf, a Laplacian-based sheaf diffusion update law can be applied asynchronously with assured convergence to the global minimizer of their Dirichlet energy. The form of asynchrony we consider allows agents to have bounded delays (where the bound can be arbitrarily large), which is sometimes called ``partial asynchrony'' in the literature~\cite{bertsekas1989parallel}.
\subsection*{Summary of Contributions}
\begin{itemize}
    \item We develop a partially asynchronous algorithm for sheaf diffusion (Algorithm~\ref{alg:sheaf_async_diffusion}).
    \item We prove that partially asynchronous nonlinear sheaf diffusion converges to the global minimizer with at least a periodic linear convergence rate under
    mild conditions (Theorem~\ref{thm:linear_conv}).
    \item We provide numerical simulations demonstrating this convergence rate (Section~\ref{sec:sim}).    
\end{itemize}

\subsection*{Related Work}
Cellular sheaves originate in algebraic topology \cite{shepard_cellular_1985,curry2014sheaves,bredon_sheaf_2012} 
and have recently emerged as a powerful modeling tool in systems engineering. Their ability to model local-to-global phenomena has found rich applications in opinion dynamics \cite{hansen2021opinion}, gossip \cite{riess_diffusion_2022}, mechanism design \cite{riess2023max}, circuit design \cite{robinson2012circuits}, and graphic statics \cite{cooperband2023towards}, just to name a few disciplines. In machine learning, sheaves and their Laplacians have served as effective message-passing operators for supervised learning on graphs \cite{hansen_sheaf_2020,bodnar2022neural,battiloro_tangent_2024,zaghen_sheaf_2024}. Our definition of ``coordination sheaves'' is heavily influenced by seminal work on homological programming \cite{hanks2025distributed, hansen_distributed_2019}, a paradigm integrating distributed optimization and homological algebra. Despite the interest in asynchronous algorithms for machine learning \cite{dudzik2024asynchronous}, existing literature on sheaf diffusion has been confined to the synchronous setting. To the best of our knowledge, this paper is the first to develop and analyze asynchronous sheaf diffusion, including over coordination sheaves.

\subsection*{Outline}
The rest of this paper is organized as follows.
Section~\ref{sec:prelim} provides preliminaries on cellular sheaves and the problem statement. Section~\ref{sec:algorithm} 
develops an asynchronous sheaf diffusion algorithm. Section~\ref{sec:conv} analyzes the convergence properties. Section~\ref{sec:sim} applies the sheaf diffusion algorithm in various asynchronous and synchronous settings.
Section~\ref{sec:discussion} concludes the paper, after a brief discussion.
(For notational conventions, please refer to Table \ref{tab:notation}.)

\begin{table}[h!] % The [h!] is a placement specifier, telling LaTeX to place the table "here!" if possible.
\centering % This centers the table on the page.
\begin{tabular}{c l} % Defines two columns: first is centered, second is left-aligned text.
\toprule % Top rule of the table
\textbf{Notation} & \textbf{Meaning} \\
\midrule % Middle rule, separating header from content
% $\R$                & The set of real numbers. \\
% $\mathbb{N}$        & The set of natural numbers (non-negative integers). \\
$\sigma_{\max}(\cdot)$ & The \textbf{largest} singular value of an operator. \\
$\sigma_2(\cdot)$   & The \textbf{smallest nonzero} singular value of an operator. \\
$\lambda_{\max}(\cdot)$ & The \textbf{largest} eigenvalue of an operator. \\
$\lambda_2(\cdot)$  & The \textbf{smallest nonzero} eigenvalue of an operator. \\
$\vert \cdot \vert$  & The cardinality of a set. \\
$\|\cdot\|$         & The 2-norm (Euclidean norm or $\ell_2$-norm). \\
$A^{\top}$          & The transpose of a matrix $A$. \\
$A^{+}$             & The Moore-Penrose pseudoinverse of a matrix. \\
$\nabla f$          & The gradient of a function $f$.  \\
\bottomrule % Bottom rule of the table
\end{tabular}
\caption{Nomenclature.}
\label{tab:notation} % A label for cross-referencing the table in your text.
\end{table}
\vspace{-1em}

\section{Preliminaries \& Problem Formulation}\label{sec:prelim}

In this section we first introduce coordination sheaves as a model for multi-agent systems and coordination goals. We then formulate the problem of achieving these coordination goals asynchronously.

\begin{figure}[b]
    \centering
    \includegraphics[width=0.8\linewidth]{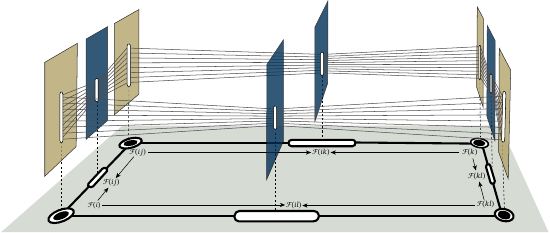}
    \caption{Mental model of a cellular sheaf: over a cycle graph, agent's stalks are tethered together with restriction maps.}
    \label{fig:celluarl-sheaf}
\end{figure}

\subsection{Cellular Sheaf Background}

We review the fundamental definitions of cellular sheaf theory (see \cite{curry2014sheaves,hansen2020laplacians,rosiak2022sheaf,ayzenberg2025sheaf}) in order to specify general classes of coordination problems over the communication graph of a system.
Recall an undirected graph consists of a pair $G = (V,E)$ where $V$ is a set of nodes and $E\subseteq V\times V$ is a set of edges. An edge between nodes $i$ and $j$ is denoted by an unordered pair of concatenated indices $ij=ji\in E$, and the neighbors of a given node is denoted $N_i \coloneqq \{j\in V\mid ij\in E\}$.

\begin{definition} \label{def:cellular-sheaf}
Given an undirected graph $G = (V,E)$, a \define{(Euclidean) cellular sheaf} $\mathcal{F}$ over $G$ is a data structure that assigns:
\begin{enumerate}
    \item[(a)] A Euclidean space $\mathcal{F}(i)$ and inner product $\langle - , - \rangle_{\mathcal{F}(i)}$ to every node $i \in V$ (\define{vertex stalk}).
    \item[(b)] A Euclidean space $\mathcal{F}(ij)$ and inner product $\langle - , - \rangle_{\mathcal{F}(ij)}$ to every edge $ij \in E$ (\define{edge stalk}).
    \item[(c)] A linear map $\mathcal{F}_{i\face ij}\maps\mathcal{F}(i) \rightarrow \mathcal{F}(ij)$ for every $i \in V$, $j \in N_i$ (\define{restriction map}).
\end{enumerate}
\end{definition}

Cellular sheaves, thus, attach spaces of various dimensions to a graph with restriction maps tethering these spaces together along edges as illustrated in Figure~\ref{fig:celluarl-sheaf}. Consider the following familiar example.

\begin{example}[Constant Sheaf]
    Given a graph $G$ and a dimension $n$, the constant sheaf, denoted $\underline{\mathbb{R}}^n$, assigns the vector space $\R^n$ to every vertex and edge stalk, and identity maps $\underline{\R}^n_{i\face ij} = \id_{\R^n}$ for every $i\in V, j\in N_i$.
\end{example}

We refer to a particular assignment $x_i \in \mathcal{F}(i)$ as a \define{local section}, and interpret this as a state for agent $i$. Edge stalks are interpreted as the communication or interaction space between agents. Collecting the data over all nodes with a direct sum of vector spaces yields the space of $0$-cochains, $C^0(G; \mathcal{F}) = \bigoplus_{i \in V} \mathcal{F}(i)$, and data over edges is similarly gathered in the space of $1$-cochains, $C^1(G; \mathcal{F}) = \bigoplus_{e \in E} \mathcal{F}(e)$. The total space $C^0(G;\FF)$ then corresponds to the global state space and $C^1(G;\FF)$ corresponds to the global communication space. These spaces, known in algebraic topology as cochain groups, are endowed with the inner products
\begin{align}
    \begin{aligned}
        \langle \mathbf{x}, \mathbf{x}' \rangle_{C^0(G;\mathcal{F})} &= \sum\nolimits_{i \in V} \langle x_i, x_i' \rangle_{\mathcal{F}(i)} \\
        \langle \mathbf{y}, \mathbf{y}' \rangle_{C^1(G; \mathcal{F)}} &= \sum\nolimits_{e \in E} \langle y_e, y_e' \rangle_{\mathcal{F}(e)}
    \end{aligned}
\end{align}
Norms are induced by these inner products in the usual way.

\begin{remark}
    Our definition of Euclidean cellular sheaves permits arbitrary inner products to be assigned to stalks, rather than just the standard $\ell_2$ inner product because of the interpretation of $\mathcal{F}(i)$ as a state space.
    For example, in a robotics system, it would be reasonable to choose the inner product on the state space $\mathcal{F}(i)$ so that $\|x_i\|^2_{\mathcal{F}(i)}$ is the Hamiltonian of the agent's dynamics.
\end{remark}

A ubiquitous problem in cellular sheaf theory is to decide whether $0$-cochains are coherent with respect to restriction maps. This consistency is expressed by the notion of a global section.

\begin{definition}\label{def:global-sections}
    Suppose $\mathcal{F}$ is a cellular sheaf over $G$. A \define{global section} is a $0$-cochain $\mathbf{x} \in C^0(G; \mathcal{F})$ such that $\mathcal{F}_{i \face ij}(x_i) = \mathcal{F}_{j \face ij}(x_j)$ for all $i \in V$ for all $j \in N_i$.
    The set of global sections is $\Gamma(G;\mathcal{F}) \subseteq C^0(G; \mathcal{F})$.
\end{definition}

Observe that the global section condition is equivalent to minimizing a quadratic potential function $(1/2)\|\FF_{i\face ij}(x_i) - \FF_{j\face ij}(x_j)\|_{\mathcal{F}(ij)}^2$ over each edge $ij \in E$. In a multi-agent system, this can be interpreted as requiring that communicating agents reach consensus on their outputs to neighboring agents. Indeed, if the vertex stalks correspond to agents' spatial positions and the constant sheaf is chosen on a connected communication graph, global sections are precisely positional consensus vectors. To generalize beyond consensus, we can alter the edge potential functions to encode different coordination goals. 

To make this precise, we need some notions from homological algebra. Define the \define{coboundary operator}
$C^0(G; \mathcal{F}) \xrightarrow{\delta_{\mathcal{F}}} C^1(G; \mathcal{F})$ given by $(\delta_{\mathcal{F}} \mathbf{x})_{ij} = \mathcal{F}_{i \face ij}(x_i) - \mathcal{F}_{j \face ij}(x_j)$.
The \define{degree-0} and \define{degree-1 cohomology} of a cellular sheaf is given by 
\begin{align}
    \begin{aligned}
        H^0(G; \mathcal{F}) &= \ker \delta_{\mathcal{F}} \\
        H^1(G; \mathcal{F}) &= C^1(G; \mathcal{F})/ \image \delta_{\mathcal{F}}
    \end{aligned}.
\end{align}
It follows immediately that $H^0(G; \mathcal{F})= \Gamma(G; \mathcal{F})$ because $x \in \ker \delta_{\mathcal{F}}$ precisely when $\mathcal{F}_{i \face ij}(x_i) - \mathcal{F}_{j \face ij}(x_j) = 0$ for all $ij \in E$. Global sections, then, can be computed as the kernel of the coboundary operator, where the relevant matrix grows in the number of nodes and edges of the graph and the dimensions of the stalks. To incorporate edge potential functions, we define the following notion of energy for a cellular sheaf.

\begin{definition} \label{def:energy-function}
    Given a cellular sheaf $\FF$ on a graph $G$ and potential functions $U_{e}\maps \FF(e)\to \R$ for each edge $e\in E$, we let the global potential function $U \maps C^1(G;\FF)\to \R$ be defined as $U(\mathbf{y})\coloneqq \sum_{e\in E} U_e(y_e)$.
    The \define{Dirichlet energy} function for $\FF$ is then $f \coloneqq U\circ \delta_\FF\maps C^0(G;\FF)\to \R$.
\end{definition}

For the following special cases, the relationship between the minimizers of the Dirichlet energy and global sections are made explicit:
\begin{enumerate}
    \item[(a)] If all potential functions are chosen to have a unique global minimizer at~$0$, then minimizers of the sheaf energy function 
    correspond to global sections:
    \[ \underset{\mathbf{x}\in C^0(G;\FF)}{\argmin} U\bigl(\delta_\FF(\mathbf{x})\bigr) = \Gamma(G;\FF).\]
    \item[(b)] If the global potential function $U$ 
    is minimized at some unique nonzero $b\in C^1(G;\FF)$ with $b\in \image \delta_\FF$, then the minimizers of the energy function correspond to an affine shift of the global sections, i.e., 
    \[\underset{\mathbf{x}\in C^0(G;\FF)}{\argmin} U\bigl(\delta_\FF(\mathbf{x})\bigr) = \delta_\FF^+b + \Gamma(G;\FF) \]
    (see \cite[Theorem 2]{hanks2025distributed}).
\end{enumerate}

A sheaf together with a family of edge potential functions provides the data of a coordination sheaf.

\begin{definition} \label{def:coordination:}
    A \define{coordination sheaf} consists of a choice of
\begin{enumerate}
    \item[(a)] communication graph $G = (V,E)$,
    \item[(b)] cellular sheaf $\mathcal{F}$ as in Definition \ref{def:cellular-sheaf},
    \item[(c)] family of edge potentials $\{U_e: \mathcal{F}(e) \to \mathbb{R}\}_{e \in E}$.
\end{enumerate}
Every coordination sheaf has an associated Dirichlet energy function~$f \coloneqq U \circ \delta_{\mathcal{F}}$ where~$U \coloneqq \sum_{e \in E} U_e$ (Definition~\ref{def:energy-function}). 
\end{definition}

\subsection{Problem Statement}

The main goal in this paper is to develop an asynchronous distributed algorithm for achieving multi-agent goals specified by coordination sheaves. We formalize this as the problem of minimizing the Dirichlet energy of a sheaf.

\begin{problem}\label{pro:1}
     Given a coordination sheaf~$(G, \FF, \{U_e\}_{e\in E})$, asynchronously solve 
     \begin{equation}\label{eq:cost_func}
         \underset{\mathbf{x}\in C^0(G;\FF)}{\text{minimize}}\: f(\mathbf{x})=U\circ\delta_\FF\left(\mathbf{x}\right).
     \end{equation}  
 \end{problem}

The key to the design of coordination sheaves is in choosing restriction maps and edge potentials such that minimizing the Dirichlet energy corresponds to achieving the desired coordination goal. The following example illustrates this philosophy with a representative application achievable by solving Problem \ref{pro:1}.

\begin{example}[Moving UAV formations]
    Consider the following coordination task for two teams of three unmanned aerial vehicles (UAVs), each with a designated leader. Each UAV is modeled by state spaces $\mathbb{R}^3 \oplus \mathbb{R}^3$ with concatenated state vectors $[p_i,v_i]^\top$ for position $p_i$ and velocity $v_i$. The task is for each team of three UAVs to maintain a fixed triangle formation by adjusting their position to achieve a specific 3D displacement vector $\hat{p}_{ij} \in \mathbb{R}^3$ relative to their leader. The formations need to move towards a target and, tracking the target, the leaders of each team align their velocity vectors together. The following coordination sheaf $\mathcal{F}$ encodes this task:
    \begin{itemize}
        \item The graph has nodes $V = \{1,2,3,4,5,6\}$. There are edges between members of the same team as well as between the leaders:  $E = \{ 12, 13, 23, 45, 46, 56, 14\}$.
        \item The node stalks are state spaces $\mathcal{F}(i) = \mathbb{R}^3 \oplus \mathbb{R}^3$. The edge stalks are $\mathcal{F}(ij) = \mathbb{R}^3$ (see Fig.~\ref{fig:coord-sheaf}).
        \item The restriction maps are projections (see Fig.~\ref{fig:coord-sheaf}).
        \item The edge potentials are $U_e(y_e) = (1/2)\| y_e - \hat{p}_{ij} \|^2$ between leader and followers and $U_e(y_e) = (1/2) \|y_e\|^2$ between leaders.
    \end{itemize}
    Minimizing the Dirichlet energy function of the sheaf 
    \begin{equation}
        f(\mathbf{x}) = \frac{1}{2}\| v_{1} - v_{4} \|^2 + \smashoperator{\sum}\nolimits_{ij \in \{12,13,45,46\}} \frac{1}{2} \| p_i - p_j - \hat{p}_{ij} \|^2
    \end{equation}
    achieves the coordination goal. Solving Problem~\ref{pro:1} would allow the agents to still achieve this goal even in the presence of uncertain communication and computation delays.
\end{example}

\begin{figure}[ht]
    \centering
    \includegraphics[width=0.9\linewidth]{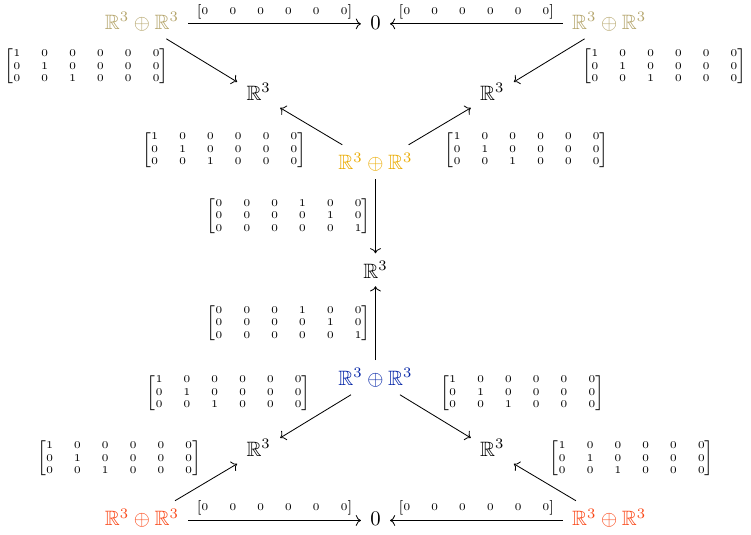}
    \caption{In the above coordination sheaf, $\mathcal{F}_{i \face ij} = \mathcal{F}_{j \face ij}$ between leaders and leaders are projection onto velocity while the restriction maps between leaders and followers are projection onto position. The restriction maps between followers are the zero map. The diagram layout is evocative of the intended formation.}
    \label{fig:coord-sheaf}
\end{figure}

%%%%%%%%%%%%%%%%%%%%%%%%%%%%%%%%%%%%%%%%%%%%%%%%%%%%%%%%%%%%%%%%%%%%
\section{Algorithm Design}\label{sec:algorithm}
%%%%%%%%%%%%%%%%%%%%%%%%%%%%%%%%%%%%%%%%%%%%%%%%%%%%%%%%%%%%%%%%%%%%

We first introduce the synchronous gradient descent algorithm for minimizing the Dirichlet energy and discuss its relation to the sheaf Laplacian. We then introduce our model of asynchrony and bring this algorithm into the asyncrhonous setting.

\subsection{Synchronous Algorithm}
For agents to minimize Dirichlet energy in a distributed fashion, we use the nonlinear sheaf Laplacian (see \cite{hansen2021opinion}) defined as follows.

\begin{definition} \label{def:sheaf-laplacian}
    Given a cellular sheaf $\FF$ on a graph $G$ with global potential $U\maps C^1(G;\FF)\to\R$, the \define{nonlinear sheaf Laplacian} is a map $L_\FF^{\nabla U}\maps C^0(G;\FF)\to C^0(G;\FF)$ defined by the gradient of the energy function, i.e.,
    \[
    L_\FF^{\nabla U}\coloneqq \delta_\FF^\top\circ \nabla U\circ\delta_\FF = \nabla(U\circ\delta_\FF).
    \]
\end{definition}

\begin{example}\label{ex:linear_L}
Suppose $U(\mathbf{y}) = \sum_{e \in E} (1/2) \| y_e\|^2$. Then, $L_\FF^{\nabla U}$ is the linear sheaf Laplacian $L_\FF=\delta_\FF^\top\delta_\FF$ (as defined in \cite{Hansen_2019_spectral}). Furthermore, for these same quadratic potentials, $L_{\underline{\mathbb{R}}}^{\nabla U}$ specializes to the standard graph Laplacian \cite{chung_spectral_1997}.
\end{example}

In the synchronous setting, a natural choice for minimizing the Dirichlet energy function of a given sheaf $\FF$ is gradient descent, which can be given in terms of the nonlinear sheaf Laplacian (see Definition \ref{def:sheaf-laplacian}) as
\begin{align}
\begin{aligned}
    \mathbf{x}(t+1) \coloneqq & \mathbf{x}(t) - \gamma L_\FF^{\nabla U}\mathbf{x}(t)
\end{aligned} \label{eq:sheaf:diffusion}
\end{align}
for a given step-size $\gamma >0$, which we view as a discrete-time dynamical system called \define{sheaf diffusion} (see \cite{bodnar2022neural}). These iterates can be computed locally by having agent~$i$ execute
\begin{equation} \label{eq:local_update}
x_i(t{+}1) = x_i(t) - \gamma \bigl[ L_{\mathcal{F}}^{\nabla U} \mathbf{x}(t) \bigr]_i
\end{equation}
for all~$i \in V$, 
 where 
 \begin{equation}
 \bigl[ L_{\mathcal{F}}^{\nabla U} \mathbf{x}(t) \bigr]_i =
\smashoperator{\sum_{j\in N_i}}\FF^\top_{i\face ij}\,\nabla U_{ij}\bigl(\FF_{i\face ij}x_i(t) - \FF_{j\face ij}x_j(t)\bigr).
\end{equation}
These dynamics are local as each agent $i$ need only know its own state $x_i(t)$ and messages $\FF_{j\face ij}x_j(t)$ from it's immediate neighbors $j\in N_i$.
Under strong-convexity assumptions on the edge potentials, the sheaf diffusion dynamics drive any initial condition $\mathbf{x}(0)$ to the orthogonal projection of $\mathbf{x}(0)$ onto $\delta_\FF^+b + H^0(G;\mathcal{F})$, where $b$ is the (necessarily) unique minimizer of the global potential function $U=\sum_{e\in E} U_e$, provided $b\in\image \delta_\FF$ \cite{hanks2025distributed}. 

\subsection{Asynchronous Algorithm}
We use~$t\in \mathbb{N}$ as a discrete-time iteration counter, and let $T_i \subseteq \mathbb{N}$ denote the iterations at which agent~$i$ performs a computation. For each $t \in \mathbb{N}$, we define $\tau_j^i(t)$ to be the block model $x_j$ that is received by agent $i$ at its $t$-th iteration.
Note that agents do not need to know $T_{i}$ or $\tau^{i}_{j}$ for
any $i$ or $j$, they are only used for analysis. To solve Problem~\ref{pro:1}, we employ a partially asynchronous sheaf diffusion algorithm.
\begin{algorithm}\label{alg:sheaf_async_diffusion}
    Given a coordination sheaf (Definition~\ref{def:coordination:}), let $x_{1}\left(0\right),\ldots,x_{N}\left(0\right)$, and $\gamma>0$ be given. For all $i\in V$, asynchronously execute
\begin{equation}\label{eq:distributed_gd}
        x_i(t+1)=
        \begin{cases}
            x_i(t),  &t\notin T_i,\\
            x_i(t)-\gamma \bigl[ L_{\mathcal{F}}^{\nabla U} \mathbf{x}^i(t) \bigr]_i, & t\in T_i,
        \end{cases}
    \end{equation}
where the local updates in \eqref{eq:distributed_gd} are computed with the $0$-cochain $\mathbf{x}^i(t) = \left( x_1\left(\tau_1^i(t)\right), x_2\left(\tau_2^i(t)\right), \dots, x_N\left(\tau_N^i(t)\right)\right)$.
\end{algorithm}
    
Two standard assumptions about asynchrony appear in the distributed computing literature (see, e.g.~\cite{bertsekas1989parallel}). The first of these is the \emph{totally asynchronous} setting where delays can be unbounded. In this setting, the iterates of each agent's local state generated by~\eqref{eq:distributed_gd} may not converge to a minimizer for consensus-type problems whose Hessians are not block diagonally dominant~\cite[Theorem 4.1]{frommer2000asynchronous}. An example is given in the standard reference~\cite[Section 7.1, Example 1.3]{bertsekas1989parallel}. For a constant sheaf, the linear sheaf Laplacian reduces to a graph Laplacian (see Example~\ref{ex:linear_L}), which is only weakly diagonally dominant and thus may not converge under total asynchrony. Therefore we only consider the \emph{partially asynchronous} setting where maximum delay bounds are assumed. 

\begin{assumption}[Partial Asynchrony]\label{assumption:partial-asynchrony}
\leavevmode
    \begin{enumerate}
        \item[(a)] There exists an integer $B \geq 0$ such that \[\{t,t+1,\cdots,t+B\} \cap T_i \neq \emptyset\] for every $i \in V$ and for every $t \in T_i$.
        \item[(b)] For all $i \in V$, $t - B \leq \tau_j^i(t) \leq t$ for all $j \in N_i$ and for all $t \in T_i$.
        \item[(c)] For all $i \in V$ and for all $t \in T_i$, $\tau_i^i(t) = t$.
    \end{enumerate}
\end{assumption}

Part (a) ensures that each agent updates at least once every~$B+1$ iterations. Part (b) ensures that communication delays are bounded above by~$B$.
Part (c) ensures that each agent always has access
to the value of its latest iterate.
In short, all delays in communications and computations
are bounded by~$B$, but~$B$ is not necessarily small.

%%%%%%%%%%%%%%%%%%%%%%%%%%%%%%%%%%%%%%%%%%%%%%%%%%%%%%%%%%%%%%%%%%%%
\section{Convergence Analysis}\label{sec:conv}
%%%%%%%%%%%%%%%%%%%%%%%%%%%%%%%%%%%%%%%%%%%%%%%%%%%%%%%%%%%%%%%%%%%%

We show that given the sheaf Dirichlet energy function $f \coloneqq U\circ \delta_\FF$ on a coordination sheaf, the iterates generated by Algorithm~\ref{alg:sheaf_async_diffusion} converge to the global minimizer. 
We start by stating assumptions on each edge potential $U_e$ as follows.

\begin{assumption}[$m_e$-Strongly Convex]\label{as:strongly_cvx}
The edge potential~$U_e$ is $m_e$ strongly convex for each $e\in E$, i.e.,
    \begin{equation}
         \langle\nabla U_e(x_e)-\nabla U_e(y_e), x_e - y_e\rangle \geq {m_e}\|x_e - y_e\|^2
    \end{equation}
    for some $m_e>0$ and all $x_e,y_e\in\FF(e)$.
\end{assumption}

\begin{assumption}[$K_e$-Smooth]\label{as:u_e_lip}
The edge potential~$U_e$ is $K_e$-smooth for each $e\in E$, i.e.,
    \begin{equation}
         \| \nabla U_e(x_e) - \nabla U_e(y_e) \|\leq K_e \|x_e - y_e\| 
    \end{equation}
    for some $K_e>0$ and all $x_e,y_e\in\FF(e)$.
\end{assumption}

We then provide the following results for $f$.

\begin{lemma}\label{lem:f_bounded}
    The sheaf Dirichlet energy $f$ (Definition~\ref{def:energy-function}) is bounded below.
\end{lemma}
\begin{proof}
    By Assumption~\ref{as:strongly_cvx}, the edge potential~$U_e$ is~$m_e$-strongly convex for each $e\in E$, meaning $U_e$ is bounded below. Thus the sum~$U(\mathbf{y})\coloneqq \sum_{e\in E} U_e(y_e)$ is also bounded below, which implies the sheaf Dirichlet energy~$f$ is bounded below.
\end{proof}

We show that the Lipschitz constant for $f$ can be obtained from the spectral properties of the linear sheaf Laplacian $L_\FF$, and the Lipschitz constants $K_e$ for each edge potential function $U_e$.

\begin{lemma}\label{lem:k_smooth}
The function $f$ (Definition~\ref{def:energy-function}) is $K$-smooth, i.e.,
\begin{equation}
     \left\|  L_\FF^{\nabla U}\mathbf{x}- L_\FF^{\nabla U}\mathbf{y}\right\| \leq K\left\| \mathbf{x}-\mathbf{y}\right\| \quad\forall \mathbf{x},\mathbf{y}\in C^0(G;\FF),
\end{equation}
where 
\begin{equation}\label{eq:lip_const}
    K=\bigl(\max_{e \in E} K_e\bigr)\cdot\lambda_{\max}(L_\FF)>0
\end{equation} 
is the Lipschitz constant.
\end{lemma}
\begin{proof}
    For any $\mathbf{x},\mathbf{y}\in C^0(G;\FF)$, we obtain
    \begin{align}
        \| L^{\nabla U}_\FF \mathbf{x} - L^{\nabla U}_\FF \mathbf{y} \|
        &= \| \delta_\FF^\top \nabla U(\delta_\FF \mathbf{x}) - \delta_\FF^\top \nabla U(\delta_\FF \mathbf{y}) \| \\
        &= \| \delta_\FF^\top \big( \nabla U(\delta_\FF \mathbf{x}) - \nabla U(\delta_\FF \mathbf{y}) \big) \| \\
        &\leq \|\delta_\FF^\top\| \| \nabla U(\delta_\FF \mathbf{x}) - \nabla U(\delta_\FF \mathbf{y}) \|, 
    \end{align}
by Assumption~\ref{as:u_e_lip}, the Lipschitz constant for the global potential is $\| \nabla U(\delta_\FF \mathbf{x}) - \nabla U(\delta_\FF \mathbf{y}) \|\leq \bigl(\max_{e \in E} K_e\bigr) \|\delta_\FF \mathbf{x} - \delta_\FF \mathbf{y} \|$, where $\max_{e \in E} K_e$ is the largest Lipschitz constant over all the edge potentials. Thus we obtain
    \begin{align}
        \| L^{\nabla U}_\FF \mathbf{x} - L^{\nabla U}_\FF \mathbf{y} \|
        &\leq \bigl(\max_{e \in E} K_e\bigr) \|\delta_\FF^\top\| \|\delta_\FF \mathbf{x} - \delta_\FF \mathbf{y}\| \\
        &\leq \bigl(\max_{e \in E} K_e\bigr) \|\delta_\FF^\top\| \|\delta_\FF\| \|\mathbf{x}-\mathbf{y}\| \\
        &= \bigl(\max_{e \in E} K_e\bigr) \sigma_{\max}^2(\delta_\FF) \|\mathbf{x}-\mathbf{y}\|,
    \end{align}
where $\|\delta_\FF^\top\| =\|\delta_\FF\|=\sigma_{\max}(\delta_\FF)$. Since
\begin{equation}
    \sigma_{\max}(\delta_\FF) = \sqrt{\lambda_{\max}(\delta^\top_\FF \delta_\FF)} = \sqrt{\lambda_{\max}(L_\FF)},
\end{equation}
therefore the Lipschitz constant is~\eqref{eq:lip_const}.
\end{proof}

\begin{remark}
The Lipschitz constant for the linear sheaf Laplacian reduces to $K=\lambda_{\max}\left( L_\FF \right)$ because all the edge potential functions are $U_e(y_e)=(1/2)\|y_e\|^2$, which is 1-Lipschitz. This result also agrees with the relevant bound for the graph Laplacian $L_{\underline{\mathbb{R}}}^{\nabla U}$ (see~\cite[Section 5.3.1]{mesbahi2010graph}).
\end{remark}

\begin{lemma}\label{lem:convexity}
The function $f$ (Definition~\ref{def:energy-function}) is convex.
\end{lemma}
 \begin{proof}
     By Assumption~\ref{as:strongly_cvx}, the edge potential~$U_e$ is $m_e$-strongly convex for each $e\in E$, which implies the global potential~$U(\mathbf{y}) = \sum_{e\in E} U_e(y_e)$ is $m$-strongly convex where $m=\min \{m_1, \dots, m_{|E|}\}$, and since $\delta_\FF$ is linear, the composition $f=U\circ\delta_\FF$ is then convex.
 \end{proof}

We denote $\XX^* = \{\mathbf{x}\in C^0(G;\FF)\mid L_\FF^{\nabla U}\mathbf{x}=0 \}$ as the minimizer set and assume it is nonempty. We then denote $f^*$ as the minimum value of $f$, i.e., $f^*:=\inf_\mathbf{x}f(\mathbf{x})$. We define the following global error bound (EB) condition, which will be used in our convergence analysis to bound the distance from agents' 
iterates to the nearest minimizer of their Dirichlet energy. Note that in general, the solution set $\mathcal{X}^*$ is a linear subspace of $C^0(G;\FF)$, meaning the nearest minimizer is always given by orthogonal projection onto this subspace.

\begin{definition}~\label{def:g_eb}
  The \define{global error bound (EB) inequality} holds if there exists a $\kappa>0$ such that
     \begin{equation}\label{eq:eb}
        \min_{\mathbf{x}^*\in\XX^*}\|\mathbf{x} - \mathbf{x}^*\|\leq\kappa\| L_\FF^{\nabla U}\mathbf{x}\|
    \end{equation} 
holds for all $\mathbf{x}\in C^0(G;\FF)$.  
\end{definition}

We next show that the Dirichlet energy function satisfies the global EB inequality, and we give
an explicit form for~$\kappa$. 

\begin{lemma}\label{lem:global_eb}
    Given a coordination sheaf in the sense of Definition~\ref{def:coordination:}, the associated Dirichlet energy $f \coloneqq U\circ\delta_\mathcal{F}$ satisfies the global error bound in the sense of Definition~\ref{def:g_eb}, with constant $\kappa=\frac{1}{m\sigma_2\left(\delta_\FF\right)}$, where $m=\min \{m_1, \dots, m_{|E|}\}$.  
\end{lemma}
\begin{proof}
    It is shown in~\cite[Appendix B]{karimi2016linear} that for a~$\varrho$-strongly convex function~$g$
    and a linear map~$x \mapsto Ax$, the function~$g(Ax)$ is convex and it satisfies the Polyak-{\L}ojasiewicz (PL)
    inequality with parameter~$=\varrho\sigma_2(A)$.  That is, $g(Ax)$ satisfies
    \begin{equation} \label{eq:gPL}
        \frac{1}{2}\|\nabla g(Ax)\|^2 \geq \varrho~\sigma_2(A)\big(g(Ax) - g^*),
    \end{equation}
    where~$g^*$ is the global minimum value of~$g \circ A$, which exists because~$g \circ A$ is convex. 
    The Dirichlet energy~$f$ is itself the composition of an~$m$-strongly convex function~$U$ with a linear map~$\delta_{\mathcal{F}}$ 
    by definition, and hence~\eqref{eq:gPL} gives
    \begin{equation} \label{eq:gPLforU}
        \frac{1}{2}\| L_\FF^{\nabla U}\mathbf{x}\|^2 \geq m\sigma_2(\delta_{\mathcal{F}})\bigl(f(\mathbf{x})-f^*\bigr). 
    \end{equation}
    The result follows from observing that if a function satisfies the PL inequality
    with parameter~$\mu$, then it satisfies the EB condition with 
    parameter~$\kappa = \frac{1}{\mu}$~\cite[Appendix A]{karimi2016linear}. 
\end{proof}

Next we define the following nonnegative quantities that measure the progress of the algorithm.
\begin{align}
\alpha\left(t\right) & \coloneqq f\big(\mathbf{x}(t)\big)-f^*, ~\label{eq:dist_cost}\\
\beta\left(t\right) & \coloneqq \sum_{\tau=t-B-1}^{t-1}\left\| \mathbf{x}\left(\tau+1\right)-\mathbf{x}\left(\tau\right)\right\|^{2}.~\label{eq:square_sum_state_diff}
\end{align}

We then obtain the following global periodic convergence result for sequences $\{\alpha(t),\beta(t)\}$ defined as in~\eqref{eq:dist_cost} and~\eqref{eq:square_sum_state_diff}.

\begin{theorem}\label{thm:linear_conv}
    Let Assumptions~\ref{assumption:partial-asynchrony} (i.e., the partial asynchrony setting),~\ref{as:strongly_cvx}, and~\ref{as:u_e_lip} hold. Given the cost~\eqref{eq:cost_func} defined in Problem~\ref{pro:1}, there exists some $\gamma_0>0$ such that if the step size satisfies $\gamma\in\left(0,\gamma_0\right)$, the sequences $\{\alpha(t),\beta(t)\}$ generated by~\eqref{eq:distributed_gd} satisfy 
    \begin{align}
        \alpha\big(r(B+1)\big) &\leq a(1-\gamma c)^r, ~\label{eq:conctraction_1}\\
        \beta\big(r(B+1)\big) &\leq b(1-\gamma c)^r ~\label{eq:conctraction_2}
    \end{align}
    for all $r\in\mathbb{N}$, where $a,b$, and $c<\frac{1}{\gamma}$ are positive constants.
\end{theorem}
\begin{proof}
Since Assumptions~\ref{assumption:partial-asynchrony} to~\ref{as:u_e_lip} hold, by Lemmas \ref{lem:f_bounded} to~\ref{lem:global_eb}, the result follows by~\cite[Theorem~9]{zhou2018distributed}.
\end{proof}

\begin{remark}~\label{rem:conv_rate}
    A smaller global EB constant $\kappa$ from~\eqref{eq:eb} provides a tighter upper bound in~\cite[Lemma 8]{zhou2018distributed}, which makes $c$ in~\eqref{eq:conctraction_1} and~\eqref{eq:conctraction_2} larger. This allows for larger contraction between iterations, i.e., faster convergence. The global EB constant $\kappa$ also depends on the spectrum of the (linear) sheaf Laplacian $L_\FF$, which we will discuss in detail later (see Section~\ref{sec:discussion}).
\end{remark}

Theorem~\ref{thm:linear_conv} implies that $\{\alpha(t),\beta(t)\}$ converges at least linearly with every $B+1$ step. Also by Lemma~\ref{lem:global_eb}, the global EB inequality holds for any $\mathbf{x}\in C^0(G;\FF)$; this implies that the linear convergence rate can be guaranteed regardless of the initial condition, i.e., the global section converges to the optimum at least linearly with every $B+1$ step, and the convergence result does not depend on initial conditions~$x_{1}\left(0\right),\ldots,x_{N}\left(0\right)$ for any agent.  

\begin{figure*}[bht]
    \centering
    \begin{minipage}{0.75\textwidth}
    \includegraphics[width=\textwidth]{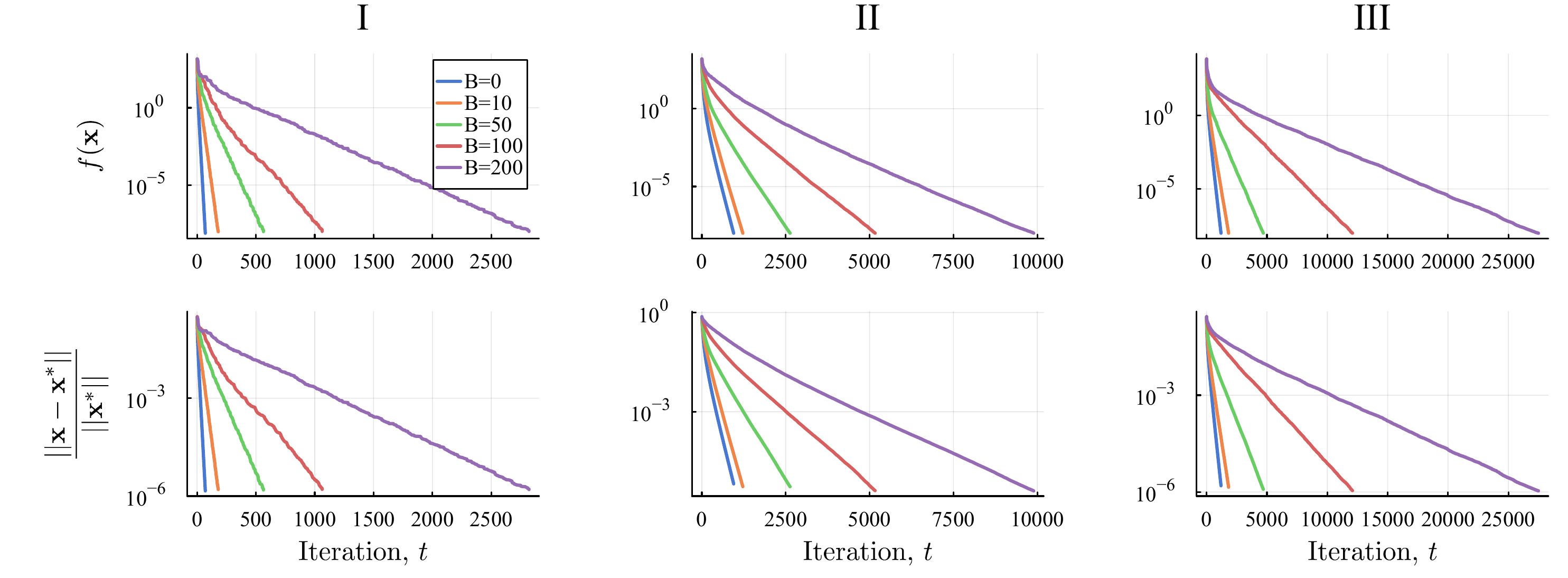}
    \end{minipage}
    \begin{minipage}{0.24\textwidth}
    \caption{Asynchronous convergence for various values of $B$. The top row plots energy while the bottom row plots relative error. The experiments were conducted on different sheaves over the same random 4-regular graph $G$. I) the constant sheaf $\underline{\R}^4$ over $G$. II) a sheaf with random restriction maps $\FF_{i\face ij}\sim\mathrm{rand}(1,4)$. III) a sheaf with random matrix-weighted edges. The positive semi-definite  matrix weights had a 0.2 probability of being strictly positive definite.}
    \label{fig:rm_experiemtn}
    \end{minipage}
\end{figure*}

\section{Numerical Results}\label{sec:sim}
We implemented several different experiments to validate our theoretical results. 
Throughout this section, we use $\mathrm{rand}(n,m)$ to denote the probability distribution of $n\times m$ matrices whose entries are drawn uniformly at random from the range $[0,1]$.

\subsection{Experimental Setup}

Primitives for constructing cellular sheaves and sheaf Laplacians are implemented in the {\tt\small AlgebraicOptimization.jl} Julia package\footnote{Code is available at~\url{https://github.com/AlgebraicJulia/AlgebraicOptimization.jl/tree/acc-experiments}}. For our experiments, we implemented a simulation of the partially asynchronous distributed computing environment as follows. 

Each agent tracks its local state as well as the (potentially stale) states of its neighboring agents. To initialize the simulation, each agent $i$ generates a local upper bound $b_i\in [B]$ on how frequently it will update its state. It also generates a time $p_i\in [b_i]$ to update its state. There is a global iteration counter $t$, and when $t \text{ mod } b_i = p_i$, agent $i$ updates its local state according to the iteration in \eqref{eq:distributed_gd}. In other words, the set $T_i$ of times when agent $i$ computes a local update is
\[
T_i=\{t_i\in \mathbb{N}\mid t_i \text{ mod } b_i = p_i\}.
\]
Similarly, each agent generates an upper bound $b_i'\in [B]$ and a time $p_i'\in [b_i']$ for how frequently it will broadcast its state to neighbors. When $t\text{ mod } b_i' =p_i'$, agent $i$ sends its most recent local state $x_i(t)$ to all neighboring agents $j\sim i$ who immediately use this value to update $\mathbf{x}^j(t)$. To make the update and broadcast schedules more random, each agent also resamples its $p_i$ and $p_i'$ values after every update and broadcast respectively.

The update bounds $b_i$ were drawn from an evenly weighted mixture of normal distributions centered at $0.05  B$ and $0.5B$. Similarly, the broadcast bounds $b_i'$ were drawn from an evenly weighted mixture of normals centered at $0.1B$ and $0.8  B$. This models a heterogeneous system with a mixture of fast and slow agents in terms of both update and broadcast times with more frequent computation than communication.

\subsection{Results}
\paragraph*{Experiment 1} We tested asynchronous sheaf diffusion for different types of sheaves and different values of the global delay bound $B$. In particular, we used a random 20 node 4-regular graph $G$ as a fixed communication topology and varied the types of restriction maps to produce different test cases. Each run of asynchronous sheaf diffusion was initialized from a fixed value of $\mathbf{x}(0)$ drawn from a Gaussian distribution centered at the origin with a variance of 10. Figure \ref{fig:rm_experiemtn} shows the sheaf Dirichlet energy and relative error in global state over time for various values of $B$ ranging from 0 to 200. In particular, Figure \ref{fig:rm_experiemtn}.III relies on the following example to construct restriction maps.
\begin{example}[Matrix-Weighted Sheaf]
    A \define{matrix-weighted graph} is an undirected graph $G$ with positive semi-definite $n\times n$ matrices $W_{ij}$ associated to each edge $ij$. These have an associated matrix-weighted Laplacian with diagonal blocks $L_{i,i} = \sum_{j \in N_i} W_{ij}$ and off-diagonal blocks $L_{i,j} = - W_{ij}$. Matrix-weighted graphs and Laplacians have found applications in formation control \cite{trinh_matrix-weighted_2018}. To any matrix-weighted graph $G$, we can construct an associated \define{matrix-weighted sheaf} $\FF$ over $G$ as follows. For each matrix weight $W_{ij}$, perform a rank-revealing QR decomposition $W_{ij}= Q_{ij}R_{ij}$ and set each restriction map $\FF_{i\face ij}=\FF_{j\face ij} = R_{ij}$. The sheaf Laplacian of $\FF$ then corresponds to the matrix-weighted Laplacian of $G$ \cite{hansen2020laplacians}.
\end{example}

\begin{figure*}[bht]
    \centering
    \begin{minipage}{0.24\textwidth}
    \caption{Running sheaf diffusion over the same sheaf $\FF$ from 100 different initial conditions with $B=50$. The sheaf is over a random 4-regular graph with random restriction maps $\FF_{i\face ij}\sim \mathrm{rand}(1,4)$. The initial conditions are sampled randomly from a Gaussian centered at the origin with variance 10.}
    \label{fig:many_initializations}
    \end{minipage}
    \begin{minipage}{0.75\textwidth}
    \includegraphics[width=\linewidth]{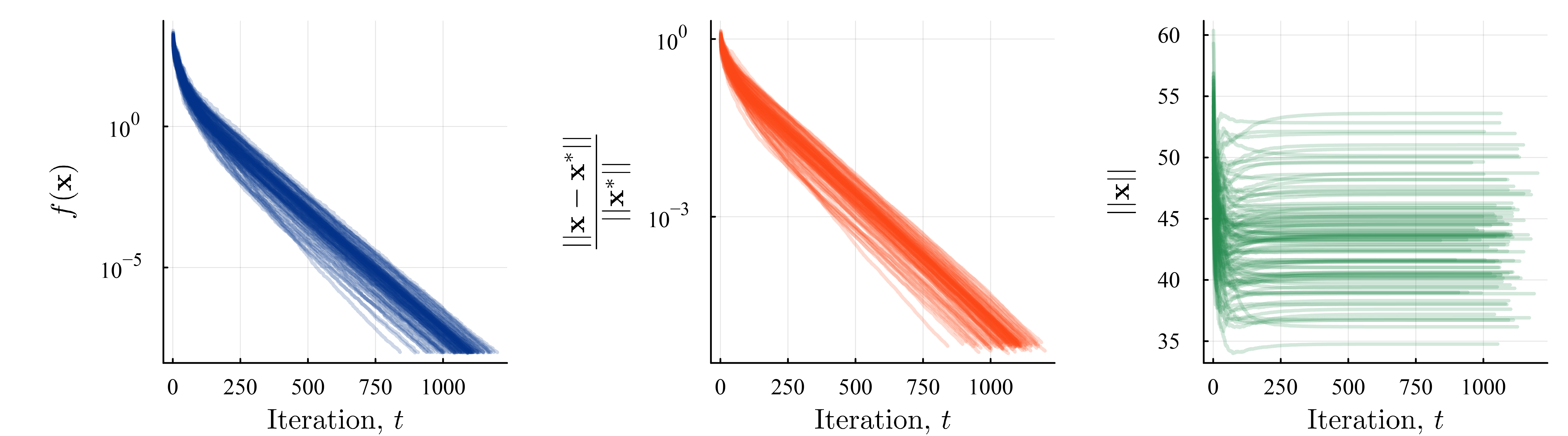}
    \end{minipage}
\end{figure*}

\paragraph*{Experiment 2} To test the global convergence properties of asynchronous sheaf diffusion, we fixed a sheaf $\FF$ over a random 20 node 4-regular graph with randomly sampled restriction maps $\FF_{i\face ij}\sim \mathrm{rand}(1,4)$. We then ran asynchronous sheaf diffusion with a fixed delay bound of $B=50$ from 100 different initial conditions sampled from a Gaussian distribution centered at the origin with a variance of 10. The energy value, relative error, and iterate norm over time are shown in Figure \ref{fig:many_initializations}.

\paragraph*{Experiment 3} In the case of synchronous sheaf diffusion, the iterates converge to the orthogonal projection of $\mathbf{x}(0)$ onto the space of global sections, which we denote as $\mathbf{x}(0)^\bot$ \cite{hansen2020laplacians}. In this experiment, we investigate the effect of asynchrony on this result. We used the same sheaf $\FF$ from Experiment 2, and ran asynchronous sheaf diffusion to convergence with various delay bounds starting from $B=0$ up to $B=2^{15}$. To mitigate the effects of random update schedules, we conducted three trials for each $B$ value and averaged the results. We then measured the distance between the average of the final iterates $x^*$ reached for each $B$ and orthogonal projection of the initial condition onto $\Gamma(G;\FF)$. This is shown in Figure \ref{fig:op_experiment}.

\paragraph*{Experiment 4} Finally, we investigated the impact of the smallest non-zero eigenvalue of the sheaf Laplacian~$\lambda_2(L_\FF)$ on convergence of asynchronous sheaf diffusion. For this experiment, we generated sheaves over Erdos-Renyi random graphs with 20 nodes and 0.3 connection probability whose restriction maps were sampled from $\mathrm{rand}(1,4)$ and computed the smallest non-zero eigenvalue of their Laplacian. We then ran asynchronous sheaf diffusion to convergence for each sheaf with a fixed delay bound of $B=50$ and recorded the number of iterations required for convergence. The scatterplot of $\lambda_2$ versus number of iterations for convergence is shown in Figure \ref{fig:lambda2_experiment}.

\begin{figure}[bht]
    \centering
    \begin{subfigure}[b]{0.24\textwidth}
        \centering
        \includegraphics[width=\textwidth]{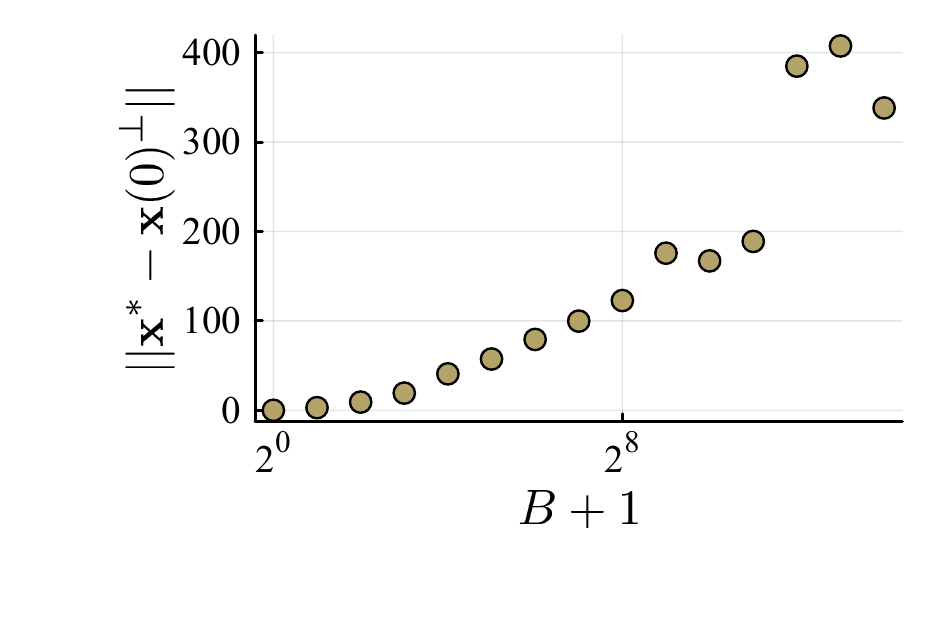}
        \caption{}
        \label{fig:op_experiment}
    \end{subfigure}\begin{subfigure}[b]{0.24\textwidth}
        \centering
        \includegraphics[width=\textwidth]{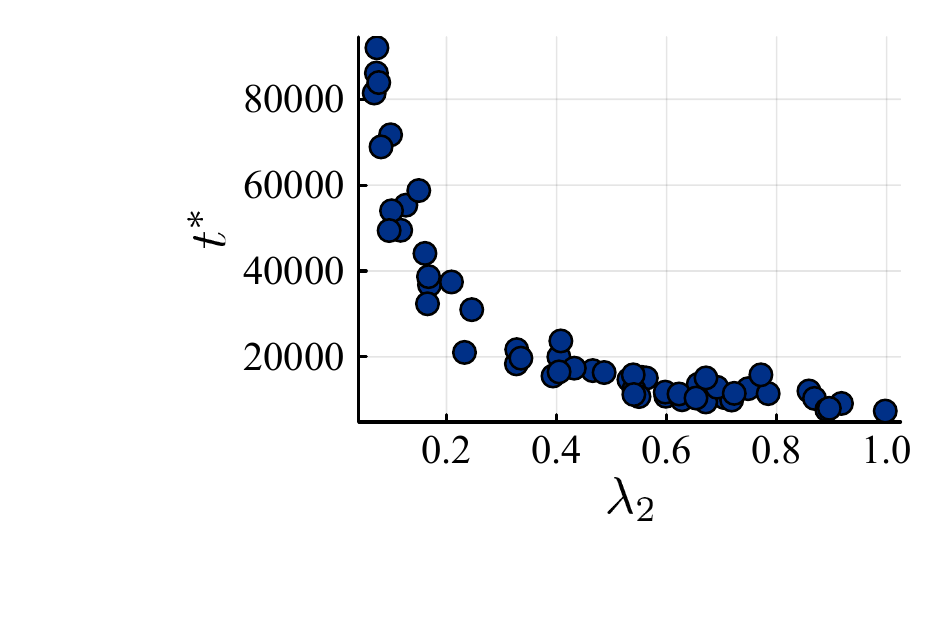}
        \caption{}
        \label{fig:lambda2_experiment}
    \end{subfigure}
    \caption{\textbf{(a)} Distance between solution obtained via asynchronous sheaf diffusion and orthogonal projection of the initial condition $\mathbf{x}(0)$ onto $\Gamma(G;\FF)$ for various values of $B$. \textbf{(b)} Smallest non-zero eigenvalue of the sheaf Laplacian ($\lambda_2$) versus number of iterations for asynchronous sheaf diffusion to reach convergence ($t^*$) for Erdos-Renyi random graphs with 20 nodes and 0.3 connection probability. In all cases, the communication bound was fixed to $B=50$.}
\end{figure}

\section{Discussion and Future Work}\label{sec:discussion}
Experimental observations are consistent with the main claim of Theorem \ref{thm:linear_conv}, namely $B+1$ step global linear convergence of asynchronous sheaf diffusion from arbitrary initial conditions. Experiment 1 demonstrates the scaling of the linear convergence rate as $B$ increases for various different types of sheaves. Importantly, testing on the constant sheaf (Figure \ref{fig:rm_experiemtn}.I) shows that our results specialize to the case of standard consensus on graphs in an asynchronous setting. Additionally, the case $B=0$ corresponds to the synchronous case, which as expected converges the most quickly. Overall, this experiment shows a consistent relationship between convergence rate and $B$-value, with larger $B$'s resulting in a slower linear convergence rate for all types of sheaves tested. 

As can be seen from Experiment 2, all trajectories converge at a linear rate regardless of initial condition. This backs up the theoretical claim of global linear convergence, rather than linear convergence only after some $\hat{t}$ as is sometimes seen in the literature \cite{tseng1991rate}. One interesting observation from Figure \ref{fig:many_initializations} is that the norm of the iterates quickly becomes constant. This suggests that our algorithm quickly converges to a sphere in $C^0(G;\FF)$ and then moves around in this sphere to reach a global section. In addition, Figure \ref{fig:op_experiment} clearly shows a positive correlation between $B$ and the distance between $\mathbf{x}^*$ and $\mathbf{x}(0)^\bot$, meaning that for larger communication and computation delays, solutions drift further from the solution obtained via synchronous computation.

The convergence of the algorithm is related to the smallest nonzero eigenvalue of the linear sheaf Laplacian $\lambda_2(L_\FF)$, as shown in Figure \ref{fig:lambda2_experiment}. This is also related to the smallest non-zero singular value of the coboundary operator 
\begin{equation}\label{eq:eigenvalue_L}
    \sigma_2(\delta_\FF) = \sqrt{\lambda_2(\delta^\top_\FF \delta_\FF)} = \sqrt{\lambda_2(L_\FF)}.
\end{equation}
From a theoretical perspective, the EB constant $\kappa$ is inversely proportional to $\sigma_2(\delta_\FF)$ (Lemma~\ref{lem:global_eb}), hence  $\kappa$ is also inversely proportional to $\lambda_2(L_\FF)$ by~\eqref{eq:eigenvalue_L}, which makes the convergence rate proportional to $\lambda_2(L_\FF)$ (Remark~\ref{rem:conv_rate}). 
This is clearly supported by Experiment 4. It is therefore desirable to design sheaves such that $\lambda_2(L_\FF)$ is as large as possible. An interesting area of future work would be to investigate how the geometric properties of the restriction maps of a cellular sheaf combine with the topological properties of the graph over which it is defined to determine how well conditioned its Laplacian is for asynchronous diffusion. The study of the spectral theory of cellular sheaves was initiated in \cite{Hansen_2019_spectral}. Other areas for future work include incorporating directed or time varying communication topologies.

\bibliographystyle{IEEEtran}
\bibliography{root}

\end{document}

%% file: tyler_macros.tex
\newcommand{\maps}{\colon}
\newcommand{\R}{\mathbb{R}}

\newcommand{\argmin}{\mathrm{argmin}}

\newcommand{\id}{\mathrm{id}}
\newcommand{\define}[1]{\textbf{#1}}

\newcommand{\FF}{\mathcal{F}}
\usepackage{stmaryrd}
\DeclareMathOperator{\face}{\trianglelefteqslant}